\documentclass{article}
\usepackage{a4}
\usepackage{amssymb,amsmath,amsthm,latexsym}
\usepackage{amsfonts}
\usepackage{graphicx}
\usepackage{paralist}
\usepackage{enumitem}
\usepackage{caption}
\usepackage[usenames]{color}
\usepackage[colorlinks=true,linkcolor=webgreen,filecolor=webbrown,citecolor=webgreen]{hyperref}%
\definecolor{webgreen}{rgb}{0,0,0.6}
\definecolor{webbrown}{rgb}{0.6,0,0}

\newtheorem{theorem}{Theorem}[section]

\newtheorem{conjecture}[theorem]{Conjecture}
\newtheorem{corollary}[theorem]{Corollary}
\newtheorem{definition}[theorem]{Definition}

\newtheorem{lemma}[theorem]{Lemma}
\newtheorem*{notation}{Notation}

\newtheorem{proposition}[theorem]{Proposition}
\newtheorem{remark}[theorem]{Remark}
\newtheorem{question}{Question}

\begin{document}  

\label{'ubf'}  
\setcounter{page}{1}                                 


\begin{center}
{ 
       {\Large \textbf { \sc 
    Combinatorial properties of sparsely totient numbers
                               }
       }
\\

\medskip

{\sc  Mithun Kumar Das, Pramod Eyyunni and Bhuwanesh Rao Patil} \\
{Harish-Chandra Research Institute,\\ HBNI, Chhatnag Road, Jhunsi, Allahabad - 211019, Uttar Pradesh, India.}\\

{\footnotesize e-mail: {\it mithundas@hri.res.in; pramodeyy@gmail.com; bhuwanesh1989@gmail.com}}
}
\end{center}

\hrulefill

\begin{abstract}  
 Let $N_1(m)=\max\{n \colon \phi(n) \leq m\}$ and $N_1 = \{N_1(m) \colon m \in \phi(\mathbb{N})\}$
where $\phi(n)$ denotes the Euler's totient function. 
Masser and Shiu  \cite{masser} call the elements of $N_1$ as `sparsely totient numbers' and initiated the study of these numbers. In this article,
we establish several results for sparsely totient numbers. First, we show that a squarefree integer divides all sufficiently large 
sparsely totient numbers and a non-squarefree integer divides infinitely many sparsely totient numbers.  Next, we construct explicit infinite families
of sparsely totient numbers and describe their relationship with the distribution of consecutive primes. We also study the sparseness of $N_1$ and prove that 
it is multiplicatively piecewise syndetic but not additively piecewise syndetic. Finally, we investigate arithmetic/geometric progressions and
other additive and multiplicative patterns like $\{x, y, x+y\}, \{x, y, xy\}, \{x+y, xy\}$ and their generalizations in the  sparsely totient numbers.

 \end{abstract}

\hrulefill




\section{Introduction}\label{introduction}
Euler's totient function $\phi(n),$ which enumerates the number of positive integers which are co-prime to 
and less than or equal to $n,$ is a classical arithmetical function. It is a well known fact that the
number of solutions to the equation $\phi(x) = m$ is finite for each $m \in \mathbb{N}$ ($\mathbb{N}$ is the set of positive integers).
It is natural, then, to ask the following question: 
 \begin{question}
For a given $m \in \mathbb{N},$ what is the largest integer $n$ such that $\phi(n) \leq m$?
 \end{question}
We denote the set $\{x : \phi(x) = m\}$ by $\phi^{-1}(m)$ and the image of $\phi$ by $V,$ i.e.
$V = \{\phi(m) : m \in \mathbb{N}\}.$ The elements of $V$ are called \textbf{totients}.
For $m \in V,$ we define the following quantities with the above question in view:
\begin{align*}
 N_1(m) &= \max\{x : \phi(x) \leq m\} \ \text{ and } \
 N_1 = \{N_1(m) : m \in V\} .
\end{align*}
Note that $N_1(m)$ can be defined on the whole of $\mathbb{N}.$
But this doesn't contribute any new elements to the image $N_1$ of $N_1(m),$ since $N_1(m) = N_1(m - 1)$ if $m \notin V.$ Hence,
from here on, we study $N_1(m)$ only for $m \in V.$ 
\begin{definition}[Sparsely totient number]
 An element of $N_1$ is called a \textbf{sparsely totient number}.
\end{definition}

 D. W. Masser and P. Shiu initiated the study of sparsely totient numbers and described  several of their properties. Here we will focus on describing
 prime divisors of sparsely totient numbers, constructing infinite families of sparsely totient numbers, understanding the  relationship of sparsely totient
 numbers with primes in short intervals and searching for interesting configurations involving addition as well as multiplication in the sparsely totient numbers.

In Section \ref{section2}, we study some divisibility properties of sparsely totient numbers. It follows from the definition of $N_1(m)$
that $N_1$ is an infinite set. 
Masser and Shiu  \cite{masser} 
showed 
that given a prime $p,$ all large enough elements of $N_1$ are divisible by $p.$ This leads to the first theorem in this paper
where we give an alternate proof of this result. 

\begin{theorem}\label{t1} 
Let $n\in \mathbb{N}.$ Then the following
are equivalent:
\begin{enumerate}[label=(\roman*)]
 \item  n is a squarefree integer.
\item There exists $M,$ depending on $n,$ such that $N_1(m) \equiv 0 \pmod{n}$ for all $m > M.$
\end{enumerate}

\end{theorem}
In their proof, Masser and Shiu studied the smallest prime not dividing $N_1(m)$ and proved that it tends to infinity as
$m \rightarrow \infty.$ To achieve this, they showed the following using the prime number theorem:
\begin{displaymath}
 \liminf_{n \in N_1} \frac{Q(n)}{\log n}\geq \sqrt{2} - 1, 
\end{displaymath}
where $Q(n)$ is the smallest prime factor not dividing $n.$ We give an alternate proof of Theorem \ref{t1} without using the prime number 
theorem. This proof  also gives some quantitative information about prime divisors of sparsely totient numbers in specific cases which is described below.
We prove the following two propositions from which the `sufficient' part
of Theorem \ref{t1} follows as a consequence. For this, we need the concept of $q$-valuations of an integer $n,$ namely, $v_q(n)$(see Definition \ref{valuation}).

\begin{proposition}\label{p1}
Suppose that $p$ and $q$ are primes such that $q>p.$ Then there exists $C_p(q)\in \mathbb{N}\cup \{0\}$ such that
for each $N\in N_1$ with $N\not\equiv 0 \pmod{p},$ we have  $v_q(N)\leq C_p(q).$ In fact, we have the following upper bounds for $C_p(q)$:
\begin{table}[ht]
\begin{center}
\begin{tabular}{|c|c|c|c|}
\hline
Prime $q>$& $p$ &$\displaystyle\frac{1}{2}\left(p+\sqrt{p^2+4(p-1)^2}\right)$&$ p(p-1) $\\ 
\hline
$C_p(q)\leq $ & $2$&$1$&$0$\\
 \hline
\end{tabular} 
\caption{}\label{table1}
\end{center}
\end{table}
\end{proposition}
\vspace*{-1cm}
\begin{proposition}\label{p2}
Suppose that $p$ and $q$ are primes such that $q<p.$ Then there exists $D_p(q)\in \mathbb{N}\cup \{0\}$ such that for
$N\in N_1$  with each prime factor less than $p,$ we have   $ v_q(N)\leq D_p(q).$
\end{proposition}

In particular, these propositions give upper bounds for the exponents of prime factors $q$ of those elements of $N_1$ which are not divisible by a
particular prime $p.$ 
 Table \ref{table1} shows that if an element of $N_1$ is not divisible by $p,$ then it cannot be divisible by any prime $q>p(p-1).$ In particular,
this gives us the
following corollary in the case when $p=3.$
\begin{corollary}\label{t1corr1}
 All sparsely totient numbers other than $2,$ are divisible by $3.$ In other words, $N_1(m)$ is divisible by $3$ for $m>1.$ 
\end{corollary}
\noindent
However, in Proposition \ref{p2} where we deal with primes $q < p,$ we have not been able to get explicit values of $D_p(q)$ as above.
We get existence of $D_{p}(q)$ in Proposition \ref{p2} due to the fact that Riemann Zeta function $\zeta(s)$ has a simple pole at $s=1$
[see Question \ref{sectionq1}].

Next,  we study  equations of the type $N_1(x)=N_1(x+1)=\dots=N_1(x+k)$ with $N_1(x-1)<N_1(x)$ and $N_1(x+k)<N_1(x+k+1)$ as an application of the sufficient part of Theorem \ref{t1} and show that $k$
can be arbitrarily large. For this, we define a subset $BN_1$ of $V$ comprising the 
``first'' elements of such  blocks $[x,x+k]\cap \mathbb{N}.$ 
$$ BN_1 := \{m \in V : N_1(m) = \max(\phi^{-1}(m))\}.$$  If $m_1$ and $m_2$ are two consecutive elements in $BN_1,$ then 
$N_1(m)=N_1(m_1)$ for each $m\in [m_1,m_2)\cap\mathbb{N}.$ We get a nice pattern in the set $BN_1$ as follows:
\begin{corollary}\label{t1corr2}
For $n\in \mathbb{N},$ $n$ divides all sufficiently large elements of $BN_1.$
\end{corollary}

The `necessary' part of Theorem \ref{t1} follows from the observation that $\prod_{q \leq p}q \in N_1.$ In Section \ref{section_construction_stn}, we generalize
this observation to construct  infinite families of sparsely totient numbers which are defined as follows:
\begin{definition}\label{family_name}
Let $\mathbb{P}$ be the set of primes and $\mathcal{O}_p=(p,\infty)\cap\mathbb{P}.$ If $p \in \mathbb{P}, n, k \in \mathbb{N}$ and $p_i$ is the $i\text{th}$ smallest prime greater
than $p,$ we define
$$X_{n, p} := \left(n\prod_{q \mid n, q \in \mathbb{P}} q^{-1}\right)\prod_{q \leq p, q \in \mathbb{P}}q, \ \ \
Y_{p, k} := \prod_{\substack{q \in [2, p_k]\cap\mathbb{P} \\ q \neq p}}q, \ \  D(p):= \frac{(p_1-p)(p_2-p)}{p(p-1)},$$
$$A(p)=\frac{p_1p_2}{p}, \ \ E(2) := \{p \in \mathbb{P} \colon \left(A(p), A(p) + D(p)\right]\cap \mathbb{P}=\varnothing\},$$
$$
E(3):=\left\{p\in \mathbb{P}: \left(\frac{p_1p_2p_3}{pq}, 1+ \frac{(p_1-1)(p_2-1)(p_3-1)}{(p-1)(q-1)}\right]\cap \mathcal{O}_q=\varnothing ~\forall q\in \mathcal{O}_p\right\},
$$
$$
\mathcal{X} := \left\{u : u=X_{n, p} \text{ for } n \in \mathbb{N}, p \in \mathbb{P}\cap \left(\frac{n}{2},\infty\right)\right\},   
\mathcal{Y}_1 := \{Y_{p, 1} : p\in \mathbb{P}\cap [5,\infty)\}, 
$$
$$
\mathcal{Y}_2 := \{Y_{p, 2} : p \in E(2) \text{ and }\ p \geq 11 \} \ \ \text{and }
\mathcal{Y}_3 := \{Y_{p, 3} : p \in E(3) \text{ and }\ p \geq 11 \}.
$$
\end{definition}

\begin{theorem} \label{t2}
The sets $\mathcal{X}$ and $\mathcal{Y}_1$ are infinite subsets of $N_1.$ Also,  $\mathcal{Y}_2$ and $\mathcal{Y}_3$ are non-empty subsets of $N_1.$ 
 \end{theorem}
 \noindent
 From Theorem \ref{t2}, we have the following general
divisibility result by observing that $X_{n,p}$ is a sparsely totient number for all $p>\frac{n}{2}$:
\begin{corollary}
Every  integer divides infinitely many sparsely totient numbers.
\end{corollary}
 Elements of $\mathcal{Y}_2$ and $\mathcal{Y}_3$ depend on the sets $E(2)$ and $E(3)$ respectively and hence the study of elements of these sets 
 is linked to the study of distribution of consecutive primes. The next theorem in this paper gives some explicit elements in the sets  
 $\mathcal{Y}_2$ and $\mathcal{Y}_3.$

\begin{theorem}\label{t3}
 For $p \in\mathbb{P}, \ p \geq 11,$
we have: \begin{enumerate}[label=(\roman*)]
\item\label{t3part1} If $p, \ p + 2i$ are consecutive primes for some $i \in \{1,2,3, 4\},$
then $Y_{p, 2} \in N_1.$ 
\item \label{t3part2}If  $p, \ p+2$ and $p+6$ are consecutive primes, then $Y_{p,3}\in N_1.$
\item \label{t3part3} Let $a, b$ be distinct positive even integers.  If $p \geq 2ab$ and 
$p, \ p+a, \ p+b$ are consecutive primes, then $Y_{p, 2} \in N_1.$
\end{enumerate}
\end{theorem}
  By Theorem \ref{t3}\ref{t3part1}, $\mathcal{Y}_2$ will be an infinite set if there is an affirmative answer to the twin prime
  conjecture which is the  case $k=2$ of the following conjecture. 
  \begin{conjecture}[Polignac's conjecture]\label{twin_prime}
   For every even natural number k, there are infinitely many pairs of primes that differ by $k.$ 
  \end{conjecture}
Our next theorem guarantees the infinitude of $\mathcal{Y}_2$ if there exist
infinitely many primes $p$ such that the fractional part 
of $\frac{p_1p_2}{p}$ is bounded above by an absolute constant $\delta_0<1.$
\begin{theorem}\label{t4}
Let $p_i$ be the $i\text{th}$ smallest prime greater than $p$ and $0 < \delta_0 <1.$ If 
$A:=\left\{p\in \mathbb{P} \colon \frac{p_1p_2}{p}-\left\lfloor\frac{p_1p_2}{p}\right\rfloor< \delta_0\right\}$ is an infinite set of primes,
then $Y_{p, 2} \in N_1$ for all but finitely many elements $p$ in $A.$
\end{theorem}
In 2014, Yitang Zhang  \cite{Zhang} proved the  weaker version of the twin prime conjecture.
\begin{proposition}[\cite{Zhang}, Yitang Zhang] \label{70000000}
Let $p_1$ be the next consecutive prime to a prime $p.$ Then there are infinitely many primes $p$ such that $p_1-p<7\times 10^{7}.$ 
\end{proposition}
\noindent
Using this, we get that $\mathcal{Y}_2$ is an infinite set as a corollary of Theorem $\ref{t4}.$
\begin{figure*}
\begin{center}
\includegraphics{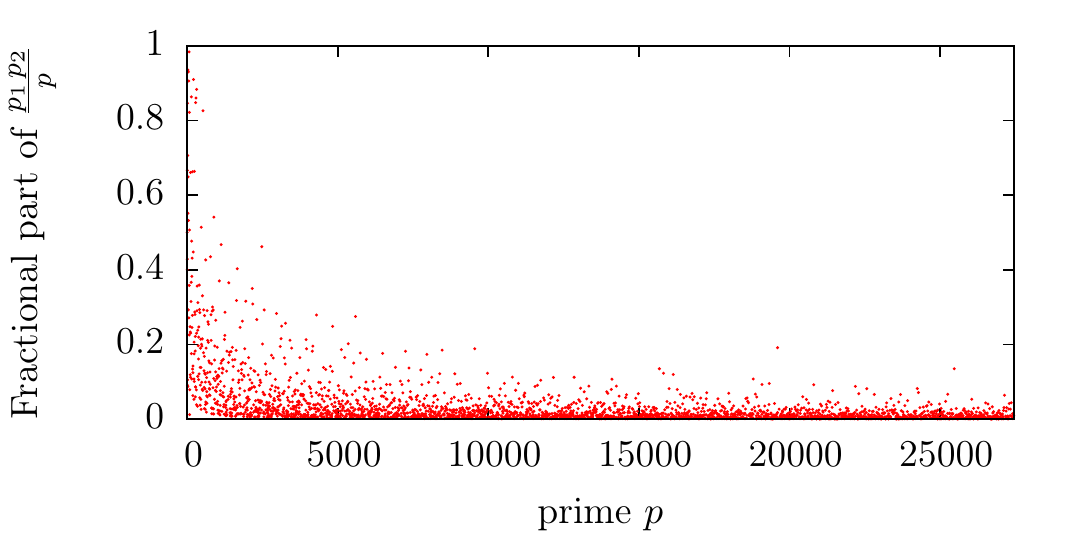}
 \caption{
 Plot of prime $p$ vs the fractional part of $\frac{p_1p_2}{p}$ where $p, p_1, p_2$ are consecutive primes.}
 \label{plot}
\end{center}
\end{figure*}
\begin{corollary}\label{Infinite_Y_2}
 $\mathcal{Y}_2$ is an infinite set.
\end{corollary}
On the basis of computations for the fractional part of $\frac{p_1p_2}{p}$ for $p\in \mathbb{P}$ in Figure \ref{plot}, one can hope that the fractional part of
$\frac{p_1p_2}{p}$ may tend to $ 0$
as $p\rightarrow \infty$ and  the fractional part of $\frac{p_1p_2}{p}$ is less than $0.983607$ for each $p\in \mathbb{P}$ [see Question \ref{sectionq3}]. If this expectation is true, then $Y_{p,2}$ is a sparsely totient 
number for all sufficiently large primes $p.$
In view of this observation we can ask the following question.
\begin{question}
 Is $\mathcal{Y}_2=\{Y_{p,2}\colon p\in \mathbb{P}, p\geq 11\}$?
\end{question}

In Section \ref{section_ampistn}, we study certain additive and multiplicative patterns in the sparsely totient numbers. More precisely,
we investigate patterns like $\{x, y, x + y\}, \{x, y, xy\}$ and generalizations of them.
$N_1$ contains the set of finite sums of arbitrarily long sequences (an $IP_0$ set)
but does not contain the set of finite sums of any infinite sequence (an $IP$ set). 
However, in the multiplicative scenario, it does contain the set of finite products of an infinite sequence (see Definition \ref{FS_FP} and \ref{IP_A_M}).
One can also ask whether both sum and product occur simultaneously in $N_1.$ This has been answered
in the negative for the subset $\mathcal{X}$ of $N_1$ [see Question \ref{sectionq5}]. We also look at the presence of arithmetic/geometric progressions
in  $N_1.$ We describe sparseness of the set $N_1$ using the notion of piecewise syndeticity [see Definition \ref{piecewise_defn}].
The following theorem  summarizes these properties: 
\begin{theorem} \label{t5}
 \begin{enumerate}[label=(\roman*)]
 \item The set of sparsely totient numbers is  multiplicatively piecewise syndetic but  not additively piecewise syndetic.
 \item The set of sparsely totient numbers is an additive $IP_0$ and  multiplicative $IP$ set but not an additive $IP$ set. 
\item There do not exist $x, y \in \mathbb{N}$ such that $x+y, xy \in \mathcal{X}.$
\item The set of sparsely totient numbers contains arbitrarily long arithmetic and  geometric progressions. 
\end{enumerate}
 \end{theorem}

 In Section \ref{section_question}, we pose some open problems arising from the present
work.

\begin{notation}
Let $\mathbb{N}, \mathbb{P}, \mathbb{Z}$ denote, respectively, the set of positive 
integers, the set of prime numbers and the set of integers. $p, q$ will always represent prime numbers
unless otherwise mentioned. 
$\lfloor x \rfloor$ denotes the greatest integer less than or equal to $x,$ $[a, b]$ denotes the set $\{x \in \mathbb{N}: a \leq x \leq b\}$ and similarly for the sets 
$(a, b], [a, b), (a, b)$ and $W(x)$ denotes the set of prime divisors of $x.$ By convention, we assume empty products and 
empty sums to take the values 1 and 0 respectively.
\end{notation}

\section{\textbf{The prime divisors of sparsely totient numbers}}\label{section2}
Euler's totient function $\phi$ is a well studied arithmetical function in number theory. For a positive integer $n,$
$\phi(n)$ counts the number of positive integers less than or equal to $n$ which are co-prime
to $n.$ It has a closed formula in terms of the prime factors dividing $n,$ namely, 
$$ \phi(n) = n \displaystyle\prod_{p \mid n}\Big(1 - \frac{1}{p}\Big).$$ From this formula, we immediately see
the following relations for $\phi$ which are frequently used in the sequel.
\begin{lemma} \cite{gupta} \label{basic.totient.bound}
Let $k, u, n \in \mathbb{N}$ and $p,p_1, p_2, \dots p_k$ be prime numbers such that $p_{i} \mid n$ for $i \in [1, k].$ Then 

   \begin{inparaenum}[(i)]
    \item \label{eularestimate}$\phi(n)\leq n\displaystyle\prod_{i=1}^{k}\displaystyle\left(1 - \frac{1}{p_{i}}\right),$
    \item \label{gupta} $\phi(pu) = \begin{cases} p\phi(u) &\mbox{if } p \mid u, \\
 (p - 1)\phi(u) &\mbox{if } p \nmid u. \end{cases}$ \\
   \end{inparaenum}
\end{lemma}

From the definition of $N_1(m),$ it follows that $N_1(m) \rightarrow \infty$ as $m \rightarrow \infty.$ 
It also follows from the definition that if $N \in \mathbb{N}$ is not in $N_1,$ then there exists a positive integer $y$ greater than $N$
with $\phi(y)\leq \phi(N)$ and our goal is to try to investigate such elements in this section. 
To prove the \textbf{\ `sufficient'} part of Theorem \ref{t1}, it is enough to prove the equivalent statement: 
\begin{equation}
 \text{Given $p\in\mathbb{P}$ and $N \in N_1$ such that $p \nmid N,$ then $N$ must be ``small enough''.} \label{trail}
\end{equation}
\noindent
We make the following definition which will be used frequently
in the sequel.
\begin{definition}[Valuation]\label{valuation}
 Let $p$ be a prime number. Then the $p$-valuation on the integers $\mathbb{Z}$ is the map
 $v_{p} \colon \mathbb{Z}\rightarrow \mathbb{N}\cup \{0, \infty\}$ defined
 by $v_{p}(0)=\infty$ and $v_{p}(n)=r$ for $n \neq 0,$ where $r$ is the largest non-negative integer such that $p^{r} \mid n.$
\end{definition}
Therefore, for a given prime $p$, we need to find $c(p)\in \mathbb{N}$ and $p_0(p)\in \mathbb{P}$  to prove 
\ref{trail} such that if  $N \in N_1$ and $p \nmid N,$ then 
\begin{enumerate}[label=(\text{\Alph*})]
\item \label{N1_A} $v_q(N)=0 \ \forall q\geq p_0(p)$ and
\item  \label{N1_B} $v_q(N)< c(p) \ \forall q<p_0(p).$
\end{enumerate}
Proposition \ref{p1} deals with
the value of $v_q(N)$ in case \ref{N1_A}  and later on,
Proposition \ref{p2} deals with the value of $v_q(N)$ in case \ref{N1_B}.
In the proof of Proposition \ref{p1} below, the element $y$ satisfying $y > N$ and $\phi(y) \leq \phi(N)$
is obtained by using the smallest multiple of $p$ greater than $q^r.$ 

We provide the statement of Proposition \ref{p1} again in full detail.
\begin{proposition}\label{p1detail}
Suppose $N$ is a positive integer and $p, q$ are primes such that $p<q,\ p \nmid N$ and
$q^r \mid N$ for some $r\in\mathbb{N}.$
\begin{enumerate}
 \item If $r>2,$ then \ $\exists \ y>N$ such that $\phi(y)\leq \phi(N).$ 
 \item  If $r=2$ and $q\geq\frac{p+\sqrt{p^2+4(p-1)^2}}{2},$ then  $\exists \ y>N$ such that $\phi(y)\leq \phi(N).$ 
 \item If $r=1$ and $q \geq (p-1)^2+p,$ then \ $\exists \ y>N$ such that $\phi(y)\leq \phi(N).$ 
\end{enumerate}
\end{proposition}
\begin{proof}
For $p,q\in\mathbb{P}$ with $p<q$ and  $r\in\mathbb{N},$ choose $k\in[1,p-1]$ such that $q^r+k\equiv 0 \pmod{p}.$
Since $k<p,$ it follows that $\gcd(q,q^r+k)=1.$ Now define
 \begin{align*}
 n:=q^r+k, \
  I_1:&=\{z\in W(N)\colon  z \mid n\}, ~
  I_2:=\{z\in W(N)\colon  z \nmid n \text{ and } z\neq q  \}. \end{align*}
Put $S = \prod_{z\in I_2}z^{v_z(N)}  \text{ and }  T=\prod_{z\in I_1}z^{v_z(N)}.$ Using the fact that $q^r \mid N,$
we get $v_q(N) \geq r.$ We set $y:=STnq^{v_q(N)-r}>N.$

Observe that $\phi(y)\leq q^{v_q(N)-r}\phi(STn)=q^{v_q(N)-r}\phi(S)\phi(Tn),$
 since $(Tn, S)=1.$
Using $I_1\subset W(n)$ along with $p \mid n$ and $p \notin I_1,$ we get
that $\phi(Tn)\leq \phi(T)\frac{p-1}{p}n$ by Lemma \ref{basic.totient.bound}. Since $n\leq q^r+p-1,$ we have $\phi(Tn)\leq \phi(T)\frac{p-1}{p}(q^r+p-1).$ Inserting this in the above estimate for $\phi(y),$ we 
find that
   \begin{equation}\phi(y)\leq q^{v_q(N)-r}\phi(ST)\left(\displaystyle\frac{p-1}{p}\right)(q^r+p-1).\label{p1detaileq1}\end{equation}                                                                                    
 \begin{equation}\displaystyle\frac{(q^r+p-1)(p-1)}{p}\leq q^{r-1}(q-1)\iff \begin{cases} (p-1)^2\leq q^{r-1}(q-p) &\mbox{if }r > 1, \\
q\geq(p-1)^2+p &\mbox{if }r=1. \end{cases}\label{p1detaileq2}\end{equation}
The right hand side of the equivalence in \eqref{p1detaileq2} is satisfied in the following cases.\\
\begin{inparaenum}[(I)]
 \item\label{p1detailcase1} $r>2,$ $q>p$ \item\label{p1detailcase2} $r=2,$ $q\geq \frac{p+\sqrt{p^2+4(p-1)^2}}{2}$ \item\label{p1detailcase3} $r=1,$ $q>p(p-1).$
\end{inparaenum}

Therefore, the inequality in \eqref{p1detaileq1} and the equivalence in \eqref{p1detaileq2} together  imply that  $\phi(y)\leq \phi(STq^{v_q(N)})=\phi(N)$ in the above cases \ref{p1detailcase1}, \ref{p1detailcase2} and \ref{p1detailcase3}. This completes the theorem.
\end{proof}
\noindent

We now proceed to prove Proposition \ref{p2} in which the lemma below will be crucial. 
\begin{lemma}
 For $p\in \mathbb{P},$  the sequence
 $\displaystyle\prod_{q\in\mathbb{P}, p<q<n}\displaystyle\frac{q}{q-1} \longrightarrow \infty \text{ as }~ n\longrightarrow \infty.$
\end{lemma}
\begin{proof}
 Let $ I_n=\{x\in \mathbb{N}\colon q \mid x, q\in\mathbb{P} \Rightarrow q<n\}$ for $n\in \mathbb{N}$ and $C=\prod_{q\in\mathbb{P}\cap[2,p]}\frac{p-1}{p}.$ 
 $$\prod_{q\in\mathbb{P}, p<q<n}\displaystyle\frac{q}{q-1}=C \prod_{q\in\mathbb{P},q<n}\displaystyle\frac{q}{q-1}=C\displaystyle\prod_{q\in\mathbb{P}, q<n}\left(\displaystyle\sum_{r=0}^{\infty}\frac{1}{q^r}\right)
 =\sum_{k\in I_n}\displaystyle\frac{C}{k }\geq \sum_{k=1}^{n-1} \displaystyle\frac{C}{k},$$ which tends to $\infty$ 
 as $ n\longrightarrow \infty.$
\end{proof}

\noindent
It can be rephrased in the following manner which is used to show Proposition \ref{p2}.
\begin{corollary} \label{N1m_div_by_p_basic_lemma3}
 Let $p\in \mathbb{P}$ and $n\in\mathbb{N}.$ Then there exists a natural number $\beta(p,n)$
 such that for each prime  $t>\beta(p,n),$  $\phi(B_t)\leq n^{-1}B_t$ with $B_t=\prod_{q\in [p,t]\cap\mathbb{P}}q.$
\end{corollary}
\begin{proof}[Proof of Proposition \ref{p2}]
Corollary \ref{N1m_div_by_p_basic_lemma3} gives
$\beta(p,q)\in\mathbb{P}$ such that 
\begin{equation}\phi(A)<q^{-1}A,\label{betapq}\end{equation}
where $A=\prod_{u\in[p,\beta(p,q)] \cap\mathbb{P}}u.$
Define $D_p(q):=\alpha$ where $\alpha$ is the unique non-negative integer satisfying $q^{\alpha}<A<q^{\alpha +1}.$ 
If $v_q(N)>\alpha,$ then choose $y:=Sq^{v_q(N)-\alpha}A,$ where  $S=Nq^{-v_q(N)}$ satisfies $\gcd(S,q)=1.$ Observe that 
 $y>N$ as $A>q^{\alpha}.$
Since the prime factors of $A$ are greater than or equal to $p,$ each prime factor of $N$ is less than $p$ and $\gcd(S,q)=1,$ it follows that the elements of $\{S,A,q\}$
are pairwise co-prime. Then
 $\phi(y)=\phi(S)\phi(q^{v_q(N)-\alpha})\phi(A).$
 Using equation \eqref{betapq} along with $A<q^{\alpha + 1}$ and Lemma \ref{basic.totient.bound}\eqref{gupta}, we get that $\phi(y)<\phi(N).$
This completes the proof.
\end{proof}
\begin{proof}[Proof of Theorem \ref{t1}]
To prove the sufficient part of the theorem, it is enough to assume that $n$ is a prime. Let $s$ be a prime and 
$s_0$ be the smallest prime greater than $(s - 1)^2 + s.$ If $N\in N_1$ such that $s\nmid N,$
then by Proposition \ref{p1detail}, $v_q(N) = 0 \ \forall \ q \geq s_0.$
Since $s < s_0,$ applying Proposition \ref{p2} with $p = s_0$ gives us $v_q(N)< D_s(q)$(depending only on $s$ and $q$)
and hence the `sufficient' part of the theorem follows.

To prove the converse, it is enough to show that $ Z_p = \prod_{q \leq p}q \in N_1$ $\forall$ $ p \in \mathbb{P}.$
Suppose for some $p,$ $Z_p \notin N_1.$ Then \ $\exists \ y > Z_p$ with $\phi(y) \leq \phi(Z_p).$ So, 
$\prod_{q \in W(y)}q^{v_q(y) - 1}(q - 1) \leq \prod_{q \leq p}(q - 1).$ Therefore $\left|W(y)\right| \leq \left|W(Z_p)\right|.$
Then the quantity
$Q:=\left(\prod_{q \in W(y)}\frac{q-1}{q}\right) \left(\prod_{q \leq p}\frac{q}{q-1}\right)\geq 1,$
because $x> p$ $\forall$ $x\in W(y)\setminus W(Z_p).$ 
It gives that
 $\frac{\phi(y)}{\phi(Z_p)} = \frac{yQ}{Z_p}\geq \frac{y}{Z_p}.$ Then 
 $y > Z_p$ implies $\phi(y)>\phi(Z_p).$ But, this is a contradiction to $\phi(y) \leq \phi(Z_p)$ and hence the theorem follows.
\end{proof}
\begin{proof}[Proof of Corollary \ref{t1corr1}]
By Proposition \ref{p1detail}, if $3 \nmid N_1(m)$ for some $m \in V,$ then $N_1(m)=2^a \cdot 5^b,$ with $a \geq 1, b \in \{0, 1\}.$
If $a>1,$ then $\phi(2^{a-1}\cdot 3 \cdot 5^b)=\phi(2^a \cdot 5^b),$ but $2^{a-1}\cdot 3 \cdot 5^b>2^a \cdot 5^b.$ Therefore,
$N_1(m)=2, 2\times 5.$ But $\phi(10)=\phi(12)$ and so $N_1(m) \neq 10.$ Hence, $N_1(m)=2,$ which implies that $m=1.$  
\end{proof}

There is a nice application of Theorem \ref{t1} to equations of the type
$N_1(m_0) = N_1(m_0 + 1) = \dots = N_1(m_0 + k)$  with $N_1(m_0 -1) < N_1(m_0)$ and $N_1(m_0 +k) < N_1(m_0+k+1).$
Recall the definition of the set $BN_1$:
$$ BN_1 = \{m \in V : N_1(m)=\max(\phi^{-1}(m))\}.$$ 
We observe that $m_0 \in BN_1.$ Note that the image of $BN_1$ under the function $h\colon x\mapsto N_1(x)$ is the whole of $N_1.$ We
enumerate the elements of $BN_1$ as $m_1, m_2, \cdots, m_t,..$ and so on.
We show that given any $n \in \mathbb{N},$ all large enough elements of $BN_1$ are divisible by $n.$ To show this, we make use 
of the following result:
\begin{lemma}( \cite[Theorem 2.17, page 88]{narkiewicz}). \label{simpledirichlet}
For any integer $n>1,$ there exists a prime $p \equiv 1 \pmod{n}.$ 
\end{lemma}

\begin{proof}[Proof of Corollary \ref{t1corr2}]
Consider the arithmetic progression $\{1+tn\}_{t \in \mathbb{N}}.$ By Lemma \ref{simpledirichlet}, there is a prime in this 
sequence, say $p,$ i.e., $1+k_0n=p$ for some $k_0 \in \mathbb{N}.$ So, $p-1=k_0n \Rightarrow n \mid (p-1).$ Using Theorem \ref{t1}, $\exists ~k_1$
such that $p \mid N_1(m_t) \ \forall \ t\geq k_1.$
So, $(p-1) \mid \phi(N_1(m_t)) = m_t \ \forall \ t \geq k_1.$ Hence $n \mid m_t$ for all $t\geq k_1.$
\end{proof}

Suppose $N\in N_1$ and there exists a prime $p, \ p \nmid N.$  
If there is a prime factor of $N$ greater than $p,$ then we have seen in Proposition \ref{p1detail} that $v_q(N) \leq 2.$
But the existence of such an $N$ in $N_1,$ which has holes in its prime factorization, is not guaranteed. In the next section, we show that there are infinitely many
elements of this type in $N_1.$ 

\section{\textbf{Explicit construction of sparsely totient numbers}}\label{section_construction_stn}
In this section, we explicitly construct several infinite families of elements in $N_1.$ As a corollary of these constructions,
we prove  a divisibility result analogous to Theorem \ref{t1} for non-squarefree integers. From the previous section, we can expect 
two types of elements, say $x$ and $y$, in $N_1.$ Here $x$ is divisible by all the primes smaller than some prime $p$ and $y$ has some ``holes''
in its prime factorization. The main tool in these proofs will be a generalization of the technique used in the proof of the 
`necessary' part of Theorem \ref{t1} (i.e., $\prod_{q\in [2,p]\cap\mathbb{P}}q\in N_1 \ \forall \ p\in \mathbb{P}).$ 

\subsection{Proof of Theorem \ref{t2}}
\begin{definition}[D(A,B)]
Let $A$ and $B$ be two finite subsets of $\mathbb{P}.$ Then 
\begin{align*}
 D(A,B):=\left(\prod_{q\in A}\frac{q-1}{q}\right)\left(\prod_{q\in B}\frac{q}{q-1}\right).
\end{align*}
\end{definition}
If $x \notin N_1,$ then $\exists \ y$ such that $y>x, \ \phi(y)\leq \phi(x).$ The following lemma gives the value of $D(W(y), W(x))$ in
this case. 
\begin{lemma} \label{Xnpmainlemma2}
 If $\phi(y)\leq \phi(x)$ and $y>x$ for  $y,x\in \mathbb{N},$ then 
 $D(W(y),W(x))$ $< 1.$ 
\end{lemma}
\begin{proof}
Using Euler product formula for $x$ and $y,$ we have $\frac{\phi(y)}{\phi(x)}=\frac{y D(W(y),W(x))}{x}.$
  Applying  $\phi(y)\leq \phi(x)$ and $y>x,$ we get $D(W(y),W(x))<1.$
  \end{proof}

For an $x \in \mathbb{N}\setminus \{1\},$ if we are able to show that $D(W(y),W(x)) \geq 1$ for all $y$ satisfying $y>x$ and $\phi(y)\leq\phi(x),$ then it means that $x \in N_1.$ The two lemmas
below help us in this regard.

\begin{lemma} \label{Xnplemma_D(A,B)1}
Let $A$ and $B$ be two finite subsets of $\mathbb{P}$ such that $|B|\leq|A|.$ 
If $\min(B\setminus A)>\max(A)$ or $B\subset A,$ then $D(B,A)\geq1.$
\end{lemma}
\begin{proof}
 For the case $B \subset A,$  $D(B, A) = \left(\prod_{q \in A\setminus B}\frac{q}{q-1}\right) \geq 1.$  
In the case when $B \not\subset A,$ one can define an injective map $f\colon B\rightarrow A$ such that $f(x)=x$ for $x\in A\cap B,$ since $|B|\leq|A|.$
Using $\min(B\setminus A)>\max(A),$ it follows that
 $f(x)\leq x~\forall ~x\in B.$ Therefore, $$D(B,A)\geq \prod_{x\in B}\frac{(x-1)f(x)}{x(f(x)-1)}=\prod_{x\in B}\frac{xf(x)-f(x)}{xf(x)-x}\geq 1.$$
 \end{proof}

\begin{lemma} \label{Xnplemma1}
 Let $A,B$ be two finite subsets of $\mathbb{P}$ such that $|A|\leq |B|$ and $A\neq B.$ If $\min(B\setminus A)>\max(A),$
 then $ \prod_{q\in B}(q-1)> \prod_{q\in A} (q-1).$
\end{lemma}
\begin{proof}
 If $A$ is a proper subset of $B,$ then the result is trivially true. On the other hand,                 
if $A\setminus B$ is non-empty, define an injective map $f\colon A\rightarrow B$ such that $f(x)=x$ for $x\in A\cap B.$ This can be
done as $|A| \leq |B|.$ Then $f(x)> x ~ \forall ~ x\in A\setminus B$ as $\min(B\setminus A)>\max(A).$ Hence
 $ \prod_{q\in B}(q-1)\geq \prod_{q\in f(A)} (q-1) = \prod_{q\prime \in A} (f(q\prime)-1)> \prod_{q\in A} (q-1).$
\end{proof}

We will now proceed to construct elements of the type $X_{n, p}$ in $N_1.$ For this purpose, we use two parameters $K(n, y)$ and $L(n, y).$ 

\begin{definition}
 Let $n,y\in\mathbb{N}.$ Then, the quantity  $K(n,y)$ is defined by $q$-valuations as $v_{q}(K(n,y)):=v_q(y),$ if $v_q(n)>0,$ and $v_{q}(K(n,y)):=0,$ otherwise.
The quantity $L(n,y)$ is  defined as follows:
\begin{align*} 
  L(n,y):=\left(\prod_{q|K(n,y)}q^{v_q(n)-v_q(y)}\right)\left(\displaystyle\prod_{q|n,  q\nmid K(n,y)}q^{v_{q}(n)-1}\right).
 \end{align*}
\end{definition} 

\begin{lemma}\label{Lnby2}
 If $n\in \mathbb{N}\setminus \{1\},$ then $L(n,y)\leq \displaystyle\frac{n}{2}$ for each $y\in \mathbb{N}.$
 \end{lemma}
\begin{proof}
For $n,y\in \mathbb{N},$ define 
  $A_1:=\{ q\in \mathbb{P}\colon q\mid n\},$
  $A_2:=\{ q\in \mathbb{P}\colon q\mid K(n,y)\}$ and 
  $A_3:=\{ q\in \mathbb{P}\colon q\mid n \text{ and } q\nmid K(n,y)\}.$
Observe that  $A_1=A_2\cup A_3.$
Since $n\geq 2,$ it follows that $A_1\neq\varnothing$ and hence $A_2\neq \varnothing$ or $A_3\neq\varnothing.$ If $A_2\neq\varnothing,$ there exists a 
 $q_0\in \mathbb{P}$ such that $v_{q_0}(K(n,y))\geq 1.$ It follows that $v_{q_0}(y)\geq 1.$ 
If $A_3\neq\varnothing,$ there exists $q_1\in A_3.$ 
Hence,
\begin{align*}
 L(n,y)&= \left(\prod_{q\in A_2}q^{v_q(n)-v_q(y)}\right)\left(\displaystyle\prod_{q\in A_3}q^{v_{q}(n)-1}\right)
 \leq \begin{cases}\displaystyle\frac{1}{q_0}\prod_{q \in A_1}q^{v_q(n)}  &\mbox{if $A_2\neq\varnothing$},\\
 \displaystyle\frac{1}{q_1}\prod_{q \in A_1}q^{v_q(n)} &\mbox{if $A_3\neq\varnothing$}.\\ \end{cases}
\end{align*}
Since $q_0, q_1 \geq 2$ and $n=\displaystyle\prod_{q \in A_1}q^{v_q(n)},$ we get $L(n,y)\leq \displaystyle\frac{n}{2}.$
\end{proof}

\begin{lemma} \label{Xnpmainlemma1}
Suppose that $p\in\mathbb{P},$ $y\in\mathbb{N}$ and $n \in \mathbb{N}\setminus \{1\}$ satisfy $p>\displaystyle\frac{n}{2}.$ Then 
 $\phi(y)\leq \phi(X_{n,p})  \Rightarrow D(W(y),W(X_{n,p}))\geq1.$ \end{lemma}
\begin{proof}
Note that  since $p>\frac{n}{2}$, we have
\begin{align*}
\phi(y)&=\displaystyle \prod_{q|K(n,y)}q^{v_q(y)-1}\displaystyle \prod_{\substack{q\nmid K(n,y) \\ q\in W(y)}}q^{v_{q}(y)-1}\displaystyle\prod_{q\in W(y)}(q-1)\text{ and }\\
 \phi(X_{n,p})&=\prod_{q|n}q^{v_{q}(n)-1}\prod_{q\in W(X_{n,p})}(q-1). 
\end{align*}

 \begin{align*}
\therefore \phi(y)\leq \phi(X_{n,p})&\Rightarrow  \prod_{\substack{q\nmid K(n,y) \\ q\in W(y)}}q^{v_{q}(y)-1}\displaystyle\prod_{q\in W(y)}(q-1)\leq L(n,y)\prod_{q\in W(X_{n,p})}(q-1)\\ 
&\Rightarrow \displaystyle\prod_{q\in W(y)}(q-1)\leq L(n,y)\prod_{q\in W(X_{n,p})}(q-1).
 \end{align*} Since $n \in \mathbb{N}\setminus \{1\},$ by Lemma \ref{Lnby2}, we have
$\displaystyle\prod_{q\in W(y)}(q-1)\leq \frac{n}{2}\prod_{q\in W(X_{n,p})}(q-1).$
 If $|W(y)|>|W(X_{n,p})|,$ then $\exists$ $r \in \mathbb{P}$ such that $r>p$ and $r\in W(y).$ This gives 
$$(r-1)\displaystyle\prod_{q\in W(y),q\neq r}(q-1)\leq \frac{n}{2}\prod_{q\in W(X_{n,p})}(q-1).$$
Applying Lemma \ref{Xnplemma1}, we get $r-1\leq \displaystyle\frac{n}{2}$ which implies that $p\leq \displaystyle\frac{n}{2},$ 
a contradiction. Therefore $|W(y)|\leq |W(X_{n,p})|.$
Applying Lemma \ref{Xnplemma_D(A,B)1}, we get the desired result.
\end{proof}
From the above lemma and Lemma \ref{Xnpmainlemma2}, we get the following proposition:
\begin{proposition}\label{p3}
 $X_{n,p}\in N_1$ for $n\in\mathbb{N}$ and $p \in \mathbb{P}$ with $p > \frac{n}{2}.$
\end{proposition}

Now, we shall investigate elements of the type $Y_{p, k},$ which have ``holes'' in their prime factorization. This is achieved 
by the series of lemmas below.

\begin{lemma}\label{Ypk_lemma4}
 If $y, k\in \mathbb{N}, p \in \mathbb{P}\cap [3,\infty)$ and $ \phi(y)\leq \phi(Y_{p,k}),$ then $|W(y)|\leq |W(Y_{p,k})|.$
\end{lemma}
\begin{proof}
 Let $B=W(y)\setminus \{p\}.$ If $B\not \subset W(Y_{p,k}),$ then $\min(B\setminus W(Y_{p,k}))>\max(W(Y_{p,k})).$ Since $ \phi(y)\leq \phi(Y_{p,k}),$ we have
$\prod_{q\in B}(q-1)\leq \prod_{q\in W(Y_{p,k})}(q-1).$ By Lemma \ref{Xnplemma1}, it follows that $|B|<|W(Y_{p,k})|$ and hence
$|W(y)|\leq |W(Y_{p,k})|.$
 If $B$ is a proper subset of $W(Y_{p,k})$ or $B=W(Y_{p,k}) \text{ with }$ $p \not\in W(y),$ then $|W(y)|\leq |W(Y_{p,k})|.$
 
 If
$B=W(Y_{p,k})$ and $p \in W(y),$ then
$\phi(y)=\prod_{q \in W(y)}q^{v_q(y)-1}(q-1)$ $\geq$ $(p-1) \phi(Y_{p, k}) > \phi(Y_{p, k})$ as $p \geq 3,$ a contradiction.
Hence, $|W(y)|\leq |W(Y_{p,k})|.$
\end{proof}

\begin{lemma}\label{Ypk_lemma5}
Suppose $y, k\in \mathbb{N}$ and $p\in \mathbb{P}$ satisfy $|W(y)|\leq |W(Y_{p,k})|.$ If there exists a prime $s\in [2,p]$ such that $s\notin W(y),$ then 
$D(W(y), W(Y_{p,k}))\geq 1 .$
\end{lemma}
\begin{proof}
 If $s=p\notin W(y),$ then $\min(W(y)\setminus W(Y_{p,k}))>\max(W(Y_{p,k}))$ or $W(y)\subset W(Y_{p,k}).$
 Then Lemma \ref{Xnplemma_D(A,B)1} along with  $|W(y)|\leq |W(Y_{p,k})|$ ensure the result in this case.
 Suppose $p\in W(y)$ and $s\in [2,p)$ be a prime such that $s\notin W(y) .$ 
 Define  $$B:=W(y)\setminus\{p\}, \ C:=W(Y_{p,k})\setminus\{s\}.$$ Observe that $|B|\leq |C|$ as $|W(y)|\leq |W(Y_{p,k})|.$
 Since $s\not \in W(y),$ it follows that  $\min(B\setminus C)>\max(C)$ or $B\subset C.$ \ 
 So, $D(B,C)\geq$ 1 by Lemma \ref{Xnplemma_D(A,B)1}. Therefore,
 \begin{align*}
D(W(y),W(Y_{p,k}))& = D(B,C)\left(\frac{s(p-1)}{(s-1)p}\right)
\geq \frac{sp-s}{sp-p} >1 ~~\text{ as } s<p.
       \end{align*}
\end{proof}
Now we shall state an important quantitative result due to Nagura on primes in short intervals to characterize elements of 
the type $Y_{p, k}$ in $N_1.$ 
\begin{proposition}[Nagura \cite{nagura}] \label{Jisturo.prime.interval}
Let $n$ be a positive integer greater than $25.$ Then the interval $(n, 1.2n)$ contains a prime.
\end{proposition}

\begin{lemma}\label{Naguralemma}
Let $p_i$ be the $i\text{th}$ smallest prime greater than $p.$ Then, $p_3 < \frac{9}{5}p \ \forall \ p \geq 11$
and $p_1 < \frac{9}{5}p \ \forall \ p \geq 5.$ 
\end{lemma}
\begin{proof}
Let $p>25$ be a prime. Then by Proposition \ref{Jisturo.prime.interval}, we have $p_1<1.2p, \ p_2<(1.2)^{2}p$ and $p_3<(1.2)^{3}p.$
Hence $p_3 < \frac{9}{5}p.$ For primes less than twenty-five, one can easily verify the truth of the lemma.
\end{proof}

From Lemma \ref{Ypk_lemma4}, $y$ has at most $k-1$ prime factors greater than $p$ in the case when $W(y)\supset [2,p]\cap \mathbb{P}.$ The following lemma describes the situation
when $y$ has exactly $k-1$ prime factors greater than $p.$ 

\begin{lemma} \label{Ypk_Lemma3}
   Let $k \in \mathbb{N}, \ k \leq 3.$ If $p \in \mathbb{P}$ and $y \in \mathbb{N}$ satisfy
  \begin{inparaenum}[(a)]
   \item $W(y)\supset [2,p]\cap \mathbb{P},$
   \item $|W(y)\cap(p,\infty)|=k-1$ and 
   \item $y$ is not squarefree,
   \end{inparaenum}
   then $\phi(y) > \phi(Y_{p, 1}) \ \forall \ p \geq 5$ and $\phi(y) > \phi(Y_{p, k}) \ \forall \ p \geq 11, k\in\{2,3\}.$
   
\end{lemma}

\begin{proof}
Let $p\in\mathbb{P}.$ Let $p_i$ be the $i\text{th}$ smallest prime greater than $p$ for $i=1,2, \dots, k$
and $(q_i)_{i=1}^{k-1}$ be $k-1$ primes greater than $p$ in increasing order such that
\begin{align*}
y&=\left(\displaystyle \prod_{q\leq p}q^{v_q(y)}\right)\prod_{i=1}^{k-1}q_i^{v_{q_i}(y)} \text{ with } ~ v_q(y), v_{q_i}(y)\geq 1 \ \text{ and }
Y_{p,k}=\left(\prod_{q<p}q\right)\prod_{i=1}^{k}p_i.
\end{align*}

If possible, let $\phi(y) \leq \phi(Y_{p, k}).$ Since $y$ is not squarefree, $\ \exists \ r \in W(y)$ such that $v_r(y)>1.$  Then 
 $ \phi(y)\leq \phi(Y_{p,k})\Rightarrow r(p-1)\prod_{i=1}^{k-1}(q_i-1)\leq \prod_{i=1}^{k}(p_i-1).$ 
This implies that $2(p-1)\leq (p_k-1)$ i.e. $p_k\geq 2p-1$ using the facts that $r\geq 2$ and $q_i \geq p_i$ for $i\in[1,k-1].$ 
Then $p_k\geq \frac{9}{5}p$ $\forall \ p\geq 5$ which contradicts  
 Lemma \ref{Naguralemma}. 
Therefore, $\phi(y) > \phi(Y_{p, 1})$ $\forall$ $p \geq 5$ and $\phi(y) > \phi(Y_{p, k})$ $\forall$ $p \geq 11,~ k\in \{2,3\}.$

\end{proof}
\begin{proposition} \label{t2Y1}
For any prime $p \geq 5, \ Y_{p, 1}\in N_1.$ 
\end{proposition}

\begin{proof}
Let $Y_{p, 1} \notin N_1$ for some $p \geq 5.$ Then, there exists 
$y\in \mathbb{N}$ such that $\phi(y)\leq\phi(Y_{p,1})$ and $y>Y_{p,1}.$ Then Lemma \ref{Xnpmainlemma2}
gives $D(W(y), W(Y_{p,1}))< 1$ and  Lemma \ref{Ypk_lemma4} gives $|W(y)|\leq |W(Y_{p,1})|.$ If there exists a prime $s \in [2, p]$ such that $s \notin W(y),$
then by Lemma \ref{Ypk_lemma5}, we have $D(W(y), W(Y_{p,1}))\geq 1,$ a contradiction.
On the other hand, if $W(y) \supset [2, p] \cap \mathbb{P},$
then Lemma \ref{Ypk_Lemma3} together with $\phi(y) \leq \phi(Y_{p, 1})$ and $|W(y)|\leq |W(Y_{p,1})|$ implies that
$y=\prod_{q \leq p}q < Y_{p, 1},$ a contradiction. 
Hence, $Y_{p, 1} \in N_1$ $\forall$ $p \geq 5.$
\end{proof}

In the next two lemmas, for $k=2, 3,$ we analyze the situation where $y$ contains at most $k-2$ prime factors greater than $p.$
\begin{lemma} \label{Ypk_Lemma1}
 For $k=2$ or $3$ and $p \geq 11,$ $(p-1)p_{k-1}p_k > p(p_{k-1}-1)(p_k-1),$
 where $p_i$ is the $i\text{th}$ smallest prime greater than $p.$
 \end{lemma}
\begin{proof}

For $k\in \{2,3\},$ Lemma \ref{Naguralemma} gives $p_k <\frac{9}{5}p$ $\forall$ $p \geq 11.$ This along with  $p_k > p_{k-1}+1,$
 gives $ \frac{p_{k-1}p_k}{p_k+p_{k-1}-1}< \frac{p_k}{2}<p, \text{ i.e., } (p-1)p_{k-1}p_k > p(p_{k-1}-1)(p_k-1).$
\end{proof}

\begin{lemma} \label{Ypk_Lemma2}
  Let $k = 2$ or $3.$ If $p \in \mathbb{P}, \ p \geq 11$ and $y \in \mathbb{N}$ satisfy
  \begin{inparaenum}[(i)]
   \item $W(y)\supset [2,p] \cap \mathbb{P}$ and
   \item $|W(y)\cap(p,\infty)|\leq k-2.$
   \end{inparaenum}
   Then we have  $\phi(y)\leq \phi(Y_{p,k})\Rightarrow y\leq Y_{p,k}. $
\end{lemma}
\begin{proof}
Let $p_i$ be the $i\text{th}$ smallest prime greater than $p$ and 
$q_1<q_2<\dots<q_l$ be the $l \ (\leq k-2)$ prime factors
of $y$ greater than $p.$ (In the case when $l=0,$ the rightmost product in the equation below is an empty product with value $1.$)
Then, $$y=\phi(y)\left(\prod_{q\leq p}\frac{q}{q-1}\right)\left(\prod_{i=1}^{l}\frac{q_i}{q_i-1}\right).$$
Inserting $\phi(y)\leq \phi(Y_{p,k})$ in the above equation, we get
$$y\leq\phi(Y_{p,k})\left(\prod_{q\leq p}\frac{q}{q-1}\right)\left(\prod_{i=1}^{l}\frac{q_i}{q_i-1}\right).$$
By the definition of $q_i$ and $p_i,$ we have $q_i\geq p_i$ for $i\in [1,l].$ Using this along with the Euler product formula for $Y_{p,k}$ in the above equation, we have
$$y\leq Y_{p,k}\left(\frac{p(p_{k-1}-1)(p_k-1)}{(p-1)p_{k-1}p_k}\right).$$
Using Lemma \ref{Ypk_Lemma1}, we obtain $y\leq Y_{p,k}.$ 
\end{proof}

 We shall see, for $k=2,$ which of the $Y_{p, k}$ are in $N_1.$ Recall $E(2), D(p)$ and $A(p)$ from Definition \ref{family_name}.
 
 \begin{proposition}\label{Yp2noprimes}
  Let $p\geq 11$ be a prime and $p\in E(2),$ then 
$Y_{p, 2}\in N_1.$
 \end{proposition}
 
 \begin{proof}
 To show $Y_{p, 2}\in N_1,$ it is enough to show that $D(W(y), W(Y_{p, 2})) \geq 1$ for $y > Y_{p, 2}$ with 
  $\phi(y) \leq \phi(Y_{p, 2})$ by Lemma \ref{Xnpmainlemma2}. 
  
  Consider $y \in \mathbb{N}$ and $p \geq 11$ satisfying $y > Y_{p, 2}$ with $\phi(y) \leq \phi(Y_{p, 2}).$ 
  Then by Lemma \ref{Ypk_lemma4}, we have $\left|W(y)\right| \leq \left|W(Y_{p, 2})\right|.$
  
  Suppose that $[2, p] \cap \mathbb{P} \subset W(y).$ Since $\left|W(y)\right| \leq \left|W(Y_{p, 2})\right|,$
  it follows that $W(y) \cap (p, \infty)$ contains at most one element.
  If $W(y) \cap (p, \infty) = \varnothing,$ then Lemma \ref{Ypk_Lemma2} gives us $y \leq Y_{p, 2},$ a contradiction. On the
  other hand, if $W(y) \cap (p, \infty)=\{q_1\}$ for some prime $q_1,$ then we have $y=\left(\prod_{q \leq p}q\right)q_1$  
  by Lemma \ref{Ypk_Lemma3} and the fact that 
  $\phi(y) \leq \phi(Y_{p, 2}).$ Since $y > Y_{p, 2}$ and $\phi(y) \leq \phi(Y_{p, 2}),$ it follows that 
  $q_1 \in (A(p), A(p) + D(p)]\cap \mathbb{P}$ and hence
  $p\not\in E(2),$ which is a contradiction.

  Therefore, $[2, p] \cap \mathbb{P} \not\subset W(y).$ Then by Lemma \ref{Ypk_lemma5}, $D(W(y), W(Y_{p, 2})) \geq 1.$
\end{proof}
Now, we move on to discuss elements of the type $Y_{p, 3}$ in $N_1.$ Recall $E(3)$ from Definition \ref{family_name}.
\begin{proposition} \label{Yp3}
  Let $p\geq 11$ be a prime and $p\in E(3),$
   then $Y_{p, 3}\in N_1.$
\end{proposition}

\begin{proof}

 To show  $Y_{p, 3}\in N_1,$ it is enough to show that $D(W(y), W(Y_{p,3})) \geq 1$ for $y > Y_{p,3}$ with 
  $\phi(y) \leq \phi(Y_{p,3}),$  by Lemma \ref{Xnpmainlemma2}. 
  
  Consider $y \in \mathbb{N}$ and $p \geq 11$ satisfying $y > Y_{p,3}$ with $\phi(y) \leq \phi(Y_{p,3}).$ 
  Then by Lemma \ref{Ypk_lemma4}, we have $\left|W(y)\right| \leq \left|W(Y_{p,3})\right|.$
  
  Suppose that $[2, p] \cap \mathbb{P} \subset W(y).$ Since $\left|W(y)\right| \leq \left|W(Y_{p,3})\right|,$
  it follows that $W(y) \cap (p, \infty)$ contains at most two elements.
  If $|W(y) \cap (p, \infty)| \leq 1,$ then Lemma \ref{Ypk_Lemma2} and $\phi(y)\leq \phi(Y_{p, 3})$ imply $y \leq Y_{p,3},$ a contradiction. On the
  other hand, if $W(y) \cap (p, \infty)=\{q_1, q_2\}$ for some distinct primes $q_1,\ q_2,$ we have $y=\left(\prod_{q \leq p}q\right)q_1q_2$ by Lemma \ref{Ypk_Lemma3} and the fact that 
  $\phi(y) \leq \phi(Y_{p,3}).$ Since $y > Y_{p,3}$ and $\phi(y) \leq \phi(Y_{p,3}),$ we get $p \not \in E(3),$ a contradiction.
  
  Therefore, $[2, p] \cap \mathbb{P} \not\subset W(y).$ Then by Lemma \ref{Ypk_lemma5}, $D(W(y), W(Y_{p,3})) \geq 1.$
\end{proof}
Recall $\mathcal{X},\mathcal{Y}_k$ for $k\in [1,3]$ from Definition \ref{family_name}. 
One can easily check that $11\in E(2)\cap E(3).$ This shows that $\mathcal{Y}_2$ and $\mathcal{Y}_3$ are non-empty families.
Therefore combining Propositions \ref{p3}, \ref{t2Y1}, \ref{Yp2noprimes} and \ref{Yp3}, we get the proof of Theorem \ref{t2}.

\begin{remark}\label{remark1}
If $a=X_{n, p} \in \mathcal{X},$ then we can rewrite this as $a=b\prod_{q \leq p}q$ where $b=n\prod_{q \in W(n)}q^{-1}$ depends only on
the value of $X_{n, p}.$ Infact, $\mathcal{X}$ can be expressed as
$\mathcal{X}=\left\{b\prod_{q \leq p}q \in \mathbb{N} \colon \ b\prod_{q \in W(b)}q<2p\right\}.$
\end{remark}

\subsection{Proof of Theorem \ref{t3}, Theorem \ref{t4} and Corollary \ref{Infinite_Y_2}}

In this section, we will proceed to find explicit elements of $\mathcal{Y}_2$ and $\mathcal{Y}_3$ and then study the infiniteness of these sets. 
Due to previous discussions, we turn our attention to the set $E(2)$ and $E(3).$ If $p \in E(2),$ then
$(A(p), A(p)+ D(p)]\cap \mathbb{P}=\varnothing.$ We shall see below that, as a consequence of the prime number theorem (Lemma \ref{prime_number_theorem}),
the length of the interval $(A(p), A(p)+ D(p)]$ becomes smaller and smaller as $p$ increases, which reduces the chances of there 
being a prime in it. Thus, one may expect that $Y_{p, 2} \in N_1$ for infinitely many $p.$

\begin{lemma} \label{prime_number_theorem}
 For $\epsilon >0,$ there is an $n_0 \in \mathbb{N}$ such that $(n, n(1+\epsilon))\cap \mathbb{P}\neq\varnothing \ \forall \ n>n_0.$
\end{lemma}

\begin{lemma}\label{Dptend0}
$D(p) \rightarrow 0$ as $ p \rightarrow \infty.$ 
\end{lemma}

\begin{proof}
Let $\epsilon > 0.$ Lemma \ref{prime_number_theorem} gives $p_0(\epsilon)\in\mathbb{N}$ such that 
$p_1 < (1 + \epsilon)p$ and $p_2 < (1 + \epsilon)^{2}p$ for all primes $p > p_0(\epsilon).$
Hence, $D(p) < 2\epsilon(\epsilon^{2} + 2\epsilon) < 6\epsilon^{2}$ and the result follows.
\end{proof}
We now give another type of sufficient conditions for $Y_{p, 2}$ to be in $N_1.$ These conditions are in terms of the distance
between successive primes.
\begin{proof}[Proof of Theorem \ref{t3}\ref{t3part1}]
 To show that $Y_{p, 2} \in N_1,$ by Proposition \ref{Yp2noprimes}, it is enough to show that 
 $\left(\frac{p_1p_2}{p}, \frac{p_1p_2}{p}+D(p)\right]\cap\mathbb{P}=\varnothing,$ where $p_i$ is the $i\text{th}$ smallest prime greater than $p.$
 If possible, let there be a prime $q$ in  $\left(\frac{p_1p_2}{p}, \frac{p_1p_2}{p}+D(p)\right].$ Then, 
 $pq = p_{1}p_2 + l \ \text{ where } l \in \left[2, \frac{(p_1-p)(p_2-p)}{p-1}\right] \cap 2\mathbb{N}.$

 We now consider several cases according to the value of $p_1 - p.$ 
 Define $d_p := \displaystyle\frac{p_1-p}{2}$ and $\Delta(d_p) := \displaystyle\frac{(p_1-p)(p_2-p)}{p-1}.$ For $p > 25,$ we get the following
 upper bounds for $\Delta(d_p)$ for various values of $d_p,$ by using Proposition \ref{Jisturo.prime.interval}.
 \begin{center}\begin{tabular}{|c|c|c|c|c|c|c|}
  \hline
  $d_p$&1&2&3&4&5&6\\
  \hline
  $\Delta(d_p) \leq$&0.94&1.88&2.82&3.76&4.70&5.64\\
  \hline
 \end{tabular}
 \end{center}
Let $d_p\in\{1, 2\}.$ If $p>25$ or  $(p,d_p)=(11,1),$ $(17,1),$ then $\Delta(d_p)<2.$
It follows that there does not exist any even $l \in [2, \Delta(d_p)]$ such that $pq = p_{1}p_2 + l,$ a contradiction.
On the other hand, if $(p,d_p)=(13,2)$, $(19,2)$, then $\Delta(d_p)<3$ and 
thus $pq = p_{1}p_2 + 2.$ We can see by simple calculation that $q$ isn't a prime here, again giving a contradiction.

Now we come to the case when $d_p=3,4.$ For $p > 25,$ we again see that $l=2$ is the only possible value in $[2, \Delta(d_p)].$
So, $pq= p_{1}p_2 + 2.$ Inserting the value $p_1=p+2d_p$ in this, we get $q= p_2 + 2\frac{d_pp_2 + 1}{p}.$ This implies that
$k_{d_p}:=\frac{d_pp_2 + 1}{p}$ is an integer. Since $p<p_2,$ we observe that $k_{d_p}>d_p.$ Applying Proposition \ref{Jisturo.prime.interval}, we get
$p_2 < 1.2p_1= 1.2(p + 2d_p).$ Hence, \begin{equation}
                                     k_{d_p}=\frac{d_pp_2 + 1}{p} < 1.2d_p + \frac{2.4d_p^2 + 1}{p}. \label{Deltad3d4}
                                    \end{equation}
Inequality \eqref{Deltad3d4} gives us $k_3 < 4 \text{ for } p > 53$ and $k_4 < 5 \text{ for } p > 197,$ which results in a
contradiction in each case since $k_{d_p}$ is an integer greater than $d_p.$ So the only remaining cases are $p=23, 31, 47, 53$ for
$d_p=3$ and $p=89$ for $d_p=4.$ For $(p, d_p)= (47, 3), (53, 3)$ and $(89, 4), \Delta(d_p) < 2,$ which gives a contradiction to the existence of $q.$
For $(p, d_p)= (23, 3), (31, 3),$ we get $\Delta(d_p) < 3$ and so $pq = p_{1}p_2 + 2.$ We can see by simple calculation that $q$
isn't a prime here.

Therefore, in each of the above cases, we get a contradiction to the existence of a prime in $(\frac{p_1p_2}{p}, \frac{p_1p_2}{p}+D(p)]$ and hence $Y_{p, 2} \in N_1.$ 
\end{proof}
Generalizing the observations in Theorem \ref{t3}\ref{t3part1} to three successive primes, we give certain qualitative results for primes
$p$ such that $Y_{p, 2}$ is in $N_1$ by means of Theorem \ref{t3}\ref{t3part3}.
\begin{proof}[Proof of Theorem \ref{t3}\ref{t3part3}]
For given  $a, b\in 2\mathbb{N},$  let $p\geq2ab$ and $p\in \mathbb{P}.$ Suppose that $p, p+a$ and $ p+b$ are consecutive primes. Clearly $p\geq 11$
  and $\frac{ab}{p}\leq \frac{1}{2}.$ It follows that
 $p+a+b<\frac{(p+a)(p+b)}{p}\leq p+a+b+\frac{1}{2} \text{ and } D(p)\leq\frac{1}{2(p-1)}.$
Therefore
$\left(\frac{(p+a)(p+b)}{p}, \frac{(p+a)(p+b)}{p} + D(p)\right]\cap \mathbb{N}=\varnothing$, since $p\geq 11.$
So $p \in E(2)$ and hence, by Proposition \ref{Yp2noprimes}, $Y_{p, 2} \in N_1.$
\end{proof}

The Theorem \ref{t3}\ref{t3part2} gives concrete examples of elements of the type $Y_{p, 3}$ in $N_1.$
\begin{proof}[Proof of Theorem \ref{t3}\ref{t3part2}]
Let $p_i$ be the $i\text{th}$ smallest prime greater than $p.$ To prove that $Y_{p,3}\in N_1,$ it is enough to show that $p\in E(3)$
by Proposition \ref{Yp3}.
 
 Suppose $p\not \in E(3),$ then there exist primes $q_1, q_2\in \mathcal{O}_p$ with $q_1<q_2$ such that
$q_2\in \left(\frac{p_1p_2p_3}{pq_1}, 1+ \frac{(p_1-1)(p_2-1)(p_3-1)}{(p-1)(q_1-1)}\right].$ 
Consider the cases on the values of $q_1.$
 \\~
 $\underline{\textbf{Case I:}}$ ($q_1\geq p_3$)\\
  Since $q_2>q_1$ and $q_1\geq p_3,$ we have $q_2\geq p_3+2.$ It follows that
  $ p_3+1\leq \frac{(p_1-1)(p_2-1)}{p-1}.$
  Setting $p_1=p+2,$ $p_2=p+6$ and $p_3\geq p_2+2$ in the above inequality, we get $(p+9)(p-1)\leq(p+1)(p+5),$ i.e., $p\leq 7,$ 
  which is a  contradiction.
\\~
 $\underline{\textbf{Case II:}}$ ($q_1=p_2$)\\
In this case,  
$q_2\in \left(\frac{p_1p_3}{p}, \frac{p_1p_3}{p}+\frac{(p_1-p)(p_3-p)}{p(p-1)}\right].$ 
By Lemma \ref{Naguralemma}, we have  $p_3<\frac{9p}{5}$ for $p\geq 11.$ Using this along with $p_1=p+2,$ we get
$q_2p\in\left(p_1p_3, p_1p_3+\frac{8p}{5(p-1)}\right).$
Since $p\geq 11,$ we have $\frac{8p}{5(p-1)}<2.$ Inserting this in the above expression, we have $q_2p=p_1p_3+1$ which is not 
possible because $q_2p$ is odd and $p_1p_3+1$ is even.
\\~
 $\underline{\textbf{Case III:}}$ ($q_1=p_1$)\\
 In this case, 
$q_2\in \left(\frac{p_2p_3}{p}, \frac{p_2p_3}{p}+\frac{(p_2-p)(p_3-p)}{p(p-1)}\right].$  
By Lemma \ref{Naguralemma}, we have  $p_3<\frac{9p}{5}$ for $p\geq 11.$ Along with $p_2=p+6,$ this gives
$q_2p\in\left(p_2p_3, p_2p_3+\frac{24p}{5(p-1)}\right).$
Since $\frac{24p}{5(p-1)}< 6$ for $p\geq 11,$ $q_2p$ can be either $p_2p_3+2$ or $p_2p_3+4.$ Since $p_2=p+6,$ we have $q_2=p_3+2d,$ where
$d$ is an integer equal to either $\frac{3p_3+1}{p}$ or $\frac{3p_3+2}{p}.$ Since $p<p_3<\frac{9p}{5}$ by Lemma \ref{Naguralemma}, we get $3<d<6.$ It follows that 
either $\frac{3p_3+1}{p}=4$ or $\frac{3p_3+2}{p}=5$ as $3p_3+1$ is an even integer. In both cases, $p\equiv 1\pmod{3}$ and hence 
$p+2\equiv 0\pmod{3}.$ This contradicts the fact that $p+2$ is a prime.

Therefore, in each case, we get a contradiction to the existence of the prime $q_1.$ Hence $p \in E(3).$
 \end{proof}

Now we shall see a criteria for the infiniteness of the set $\mathcal{Y}_2$ in 
Theorem \ref{t4} by taking a condition on the fractional part of $\frac{p_1p_2}{p}.$

\begin{proof}[Proof of Theorem \ref{t4}]
By Lemma \ref{Dptend0}, $D(p) \rightarrow 0$ as
$p \rightarrow \infty.$ So there exists a prime $p_0(\delta_0)$ such that $D(p) < 1 - \delta_0 \ \forall \ p > p_0(\delta_0).$
Therefore, for $p \in A \cap \left[p_0(\delta_0), \infty\right),$ we have $$\frac{p_{1}p_2}{p} + D(p) < \left\lfloor\frac{p_{1}p_2}{p}\right\rfloor + \delta_0 + 1 - \delta_0 =\left\lfloor\frac{p_{1}p_2}{p}\right\rfloor + 1.$$
Hence, $\left(\frac{p_{1}p_2}{p}, \frac{p_{1}p_2}{p} + D(p)\right]$ doesn't contain any integers for 
$p \in A \cap [p_0(\delta_0), \infty).$ Therefore $p \in E(2)$ for each $p$ in $A \cap [p_0(\delta_0), \infty).$ Hence,
by Proposition \ref{Yp2noprimes}, $Y_{p, 2} \in N_1$ for $p \in A \cap [p_0'(\delta_0), \infty),$ where $p_0'(\delta_0)=\max\{11,p_0(\delta_0)\}.$
\end{proof}

We are now going to show that $\mathcal{Y}_2$ is an infinite set 
using the result on Bounded gaps between consecutive primes by Yitang Zhang (see Proposition \ref{70000000}). 
 \begin{proof}[Proof of Corollary \ref{Infinite_Y_2}]
 Let $\epsilon <1$ be a positive real number and $C=7\times 10^7.$ Let $p,p_1 $ and $p_2$ be consecutive primes for each $p\in\mathbb{P}.$
 By prime number theorem, there exists $p_0(\epsilon)\in \mathbb{N}$ such that $p_2-p<\frac{\epsilon p}{C}$ $\forall$ $p\geq p_0(\epsilon).$ Consider the set
 $$R=\{p\in \mathbb{P}\colon p_1-p<C \text{ and } p\geq p_0(\epsilon)\}.$$ By Proposition \ref{70000000}, $R$ is an infinite set. 
 
 Let $p\in R.$ Then $p_1=p+N$ for some $N<C.$ It gives us  $\frac{p_1p_2}{p}=p_2+N+\frac{N(p_2-p)}{p}$ where
$ \frac{N(p_2-p)}{p}<\frac{N\epsilon}{C}<\epsilon$ as $p_2-p<\frac{\epsilon p}{C}$ for $p\geq p_0(\epsilon)$ and $N<C.$ Hence  the fractional part of $\frac{p_1p_2}{p}$ is less than $\epsilon$ for each $p\in R.$ Then
Theorem \ref{t4} ensures that $Y_{p,2}$  is a sparsely totient number for all sufficiently large primes $p\in R.$ Since $R$ is an infinite set, $\mathcal{Y}_2$ is an infinite set.
 \end{proof}


\section{\textbf{Additive and multiplicative patterns in sparsely totient numbers}}\label{section_ampistn}
We now turn our attention to additive and multiplicative configurations like sets of finite sums, sets of finite products
and arithmetic and geometric progressions inside $N_1.$ Below, we give the relevant definitions.

Let $G$ be a non-empty set and let $P_f(G)$  be the collection of all non-empty finite subsets of $G.$ 
\begin{definition}\label{piecewise_defn}
 Let $(G,+)$ be a commutative semigroup and $A\subset G.$ Then 
 \begin{enumerate}[label=(\alph*)]
  \item $A$ is syndetic iff there exists $S\in P_f(G)$ such that $G=\displaystyle\cup_{t\in S}(-t+A).$
  \item $A$ is thick iff whenever $F\in P_f(G),$ there exists $x\in G$  such that $F+x\subset A.$
  \item $A$ is piecewise syndetic iff $\exists$ $S\in P_f(G)$ such that $\displaystyle\cup_{t\in S}(-t+A) \text{ is thick}.$
\end{enumerate}
   \end{definition}
   
If $G=\mathbb{N}$ with binary operation addition $+(\text{or } \cdot),$ then syndetic, thick and piecewise syndetic sets are called, respectively,
  additively (or multiplicatively) syndetic, additively (or multiplicatively) thick and additively (or multiplicatively) piecewise syndetic.
 
\begin{proposition}\label{piecewise_syndetic_N_1}
 $N_1$ is  multiplicatively piecewise syndetic but  not additively piecewise syndetic.
\end{proposition}

\begin{proof}
 Suppose $S\in P_f(\mathbb{N}).$  
Let $u=\max(S).$ Suppose $p$ is a prime with $p>u.$ Then by Theorem \ref{t1}, \ $\exists \ n_0(p) \in N_1$ such that 
\ $\forall \ n_1>n_2>n_0, \ n_1, n_2 \in N_1,$ we have $p \mid (n_1-n_2).$ In particular, $n_1-n_2 \geq p.$
Observe that the set $\left(\cup_{t=1}^u(-t + N_1)\right)\cap [n_0(p), \infty)$ does not contain any interval of length greater than $p.$
This means that $\left(\cup_{t=1}^u(-t + N_1)\right)$ is not additively thick  and therefore its subset $\left(\cup_{t \in S}(-t + N_1)\right)$
is also not additively thick. Therefore, $N_1$ is not additively piecewise syndetic.

We will prove that $N_1$ is multiplicatively piecewise syndetic  by showing that its subset $\mathcal{X}$ itself has this property.
It is enough to show that $2^{-1}\mathcal{X}=\{n\in \mathbb{N}\colon 2n\in \mathcal{X}\}$ is multiplicatively thick. 
Let $F$ be a finite set. Define $r=\max\{\lambda_u \colon u \in F\},$ where $\lambda_u=2u\prod_{q\in W(2u)}q.$

Let $p\in\mathbb{P}$ and $p>r.$
Then,  $X_{\lambda_u, p}=2u\prod_{q\leq p, q\in \mathbb{P}}q \ \forall \ u \in F.$ Since $\lambda_u\leq r<p,$ we have
$X_{\lambda_u, p} \in \mathcal{X}.$ 
Hence,  
$F\cdot\displaystyle\prod_{q\leq p, q\in \mathbb{P}}q \subset 2^{-1}\mathcal{X}$ and so $2^{-1}\mathcal{X}$ is a thick set. 
 \end{proof}

\begin{definition}[Finite sums and Finite products]\label{FS_FP}
  Let $J$ be a subset of   $\mathbb{N}$ and suppose $x=(x_n)_{n \in J}$ is a sequence
  indexed by this subset. Then we define the set of Finite sums $FS(x)$ and the set of Finite products $FP(x)$ as follows:
  $$FS(x)=\left\{\displaystyle \sum_{n\in F}x_n\colon F\in P_f(J)\right\} ~\text{ and }~
  FP(x)=\left\{\displaystyle \prod_{n\in F}x_n\colon F\in P_f(J)\right\}.$$
  \end{definition}
  \begin{definition}\label{IP_A_M}
 Let $A\subset \mathbb{N}$ and $r\in\mathbb{N}.$ Then
 \begin{enumerate}[label=(\alph*)]
  \item 
$A$ is called an additive $IP_r$ set (respectively, multiplicative $IP_r$ set) if there is
 a sequence $x=(x_n)_{n\in J}$ with $|J|=r$ such that $FS(x)\subset A$ $(FP(x) \subset A).$ 
 \item $A$ is called an additive $IP_0$ set 
 (respectively, multiplicative $IP_0$ set) if $A$ is an additive $IP_r$ set (multiplicative $IP_r$ set) for each $r\in \mathbb{N}.$
 \item $A$ is called an additive $IP$ set (respectively, multiplicative $IP$ set) if there exists an infinite sequence 
 $x=(x_n)_{n\in \mathbb{N}}$  such that $FS(x)\subset A$ $(FP(x) \subset A).$ 
 \end{enumerate}
\end{definition}

First, we study the existence of an additive $IP_r$ set in $N_1.$ For example, if $r=2,$ by an additive $IP_2$ set, we mean a set which contains a 
triplet of the form $\{x, y, x+y\}.$ In the following proposition, we show that $\mathcal{X}$ contains the
set of finite sums of arbitrarily long finite sequences.  

\begin{proposition}\label{XIp0}
$\mathcal{X}$ is an additive $IP_{0}$ set.
\end{proposition}
\begin{proof}
It is enough to show that $\mathcal{X}$ is an additive $IP_r$ set for each $r\in \mathbb{N}.$ Let $r\in\mathbb{N}$  and 
$\{n_i\colon i\in [1, r]\}\subset \mathbb{N}.$ 
Define
$m_i:={n_i}{\prod_{q\in W(n_i)}q^{-1}}$  and  $u_{I}=\sum_{i\in I}m_i$ $\forall$ $i\in[1,r]$ and $I\subset [1,r].$

Let $p$ be a prime such that $p> \left(\sum_{i=1}^{r}m_{i}\right)^2$ and 
consider the finite sequence $E=\{X_{n_i,p}\}_{i=1}^{r}.$ We will show that $\mathcal{X}$ is an additive $IP_r$ set by showing
$FS(E)\subset \mathcal{X}.$
 
 Let $P$ be the product of primes less than or equal to $p.$ Then $X_{n_i,p}= m_iP.$ It gives $\sum_{i\in I}X_{n_i,p}=\left( \sum_{i\in I}m_i\right)P=u_{I}P$ for  $I\subset [1,r].$
So, $\sum_{i\in I}X_{n_i,p}=X_{b_I,p},$ where $b_I=u_{I}\prod_{q \in W(u_I)}q \leq u_{I}^{2}<p, \text{ as } I \subset [1, r].$
Thus, $X_{b_I,p}\in \mathcal{X}.$ Hence, $FS(E)\subset \mathcal{X}.$
\end{proof}

However, as the next proposition shows, we see that $N_1$ does not contain the set of finite sums of any infinite sequence via an application
of Theorem \ref{t1}.

\begin{proposition}\label{additive_IP}
 $N_1$ is not an additive IP set.
\end{proposition}
\begin{proof}
 If possible, suppose $N_1$ is an additive $IP$ set. Then there exists an increasing infinite sequence $x=(x_n)_{n\in \mathbb{N}}$ in $ N_1$ such that
 $ FS(x)\subset N_1 .$ 
 Set $y_n= \sum_{i=1}^{n}x_i $ and $z_n= \sum_{i=2}^{n}x_i$ for $n\geq 2.$ Since $FS(x)\subset N_1,$ it follows that 
 $y_n,z_n\in N_1$ and $y_n-z_n=x_1$ $\forall$ $n\geq 2.$
 
 Suppose $p$ is a prime greater than $x_1.$ Then by Theorem \ref{t1}, there exists $u_0\in \mathbb{N}$ such that $u\geq u_0$  and 
 $u\in N_1\Rightarrow u\equiv 0 \pmod{p}.$ Since $(y_n) $ and $(z_n)$ are strictly increasing sequences,
 there exists $n_0\in \mathbb{N}$ such that $y_n\geq u_0$ and $z_n\geq u_0$ for all $n\geq n_0.$
 Since $y_n, z_n \in N_1$ for all $n \geq \max\{n_0, 2\},$ it follows that $x_1=y_n-z_n \equiv 0\pmod{p},$ a contradiction to
 $0<x_1<p.$ Therefore $N_1$ is not an additive $IP$ set.
\end{proof}
\noindent
But, $N_1$ does contain the set of finite products of an infinite sequence as shown below. 
\begin{proposition}\label{multiplicative_IP}
$\mathcal{X}$ is a multiplicative $IP$ set.
\end{proposition}
\begin{proof}
 Define an increasing sequence of primes $(p_n)_{n\in\mathbb{N}}$ and an increasing sequence $(x_n)_{n\in \mathbb{N}}$ in $\mathcal{X}$ 
 such that  
 \begin{align*}
  x_1&=X_{2, 2}=2, \ p_1=2,\\
  x_n&=X_{2,p_n}, \text{ where } p_n>x_{n-1}^n \text{ for } n>1. 
 \end{align*}
Let $E=FP((x_n)_{n\in\mathbb{N}}).$ To prove that $\mathcal{X}$ is a multiplicative $IP$ set, 
it suffices to show that $E\subset \mathcal{X}.$ 
 Suppose $\prod_{i=1}^{l}x_{k_i}\in E,$ where $k_i<k_{i+1}$ $\forall$ $i\in [1,l).$ If $k_l=1,$ then $l=1$ and $\prod_{i=1}^{l}x_{k_i}=x_1=X_{2, 2} \in \mathcal{X}.$
 On the other hand, if $k_l > 1,$ we have  
  $\displaystyle\prod_{i=1}^{l}x_{k_i}=X_{b,p_{k_l}}$ with
$b=X_{2,p_{k_1}}^l\displaystyle\prod_{i=2}^{l-1} \left(\displaystyle\frac{X_{2,p_{k_i}}}{X_{2,p_{k_{i-1}}}}\right)^{l+1-i} \leq X_{2,p_{k_{l-1}}}^{l}=x_{k_{l-1}}^l.$ 
  Since $k_i<k_{i+1}$ $\forall$ $i\in [1,l),$ we get that $ k_{l-1}\leq k_l-1$
 and $k_l \geq l.$ 
Then $b\leq x_{k_l-1}^{k_l}$ as $(x_n)_{n\in \mathbb{N}}$ is an increasing sequence. Since $k_l>1,$ we get $b<p_{k_l}$ and hence
 $\displaystyle\prod_{i=1}^{l}x_{k_i}=X_{b,p_{k_l}}\in \mathcal{X}.$ This completes the proof.
 \end{proof}
 
Next, we would like to know whether additive and multiplicative configurations can occur together in $N_1.$ 
Even for the simplest configuration $\{x+y, xy\}$ for $x, y \in \mathbb{N},$ we do not know the answer to this.
However, we prove that $\mathcal{X}$ doesn't contain such a pair.
\begin{proposition}\label{x+y_xy}
The set  $\mathcal{X}$ does not contain a configuration of the type $\{x+y,xy\}$ for any positive integers $x, y.$
  \end{proposition}
\begin{proof}
 If possible assume that $x+y, xy \in \mathcal{X}$ for some $x, y \in \mathbb{N}.$ Let $x+y=X_{n_1,p_1}$ and 
 $xy=X_{n_2,p_2}$ for some $p_1,p_2\in \mathbb{P}$ and $n_1,n_2\in\mathbb{N}$ satisfying $p_1>\frac{n_1}{2}$ and $p_2>\frac{n_2}{2}.$

 Since $xy=X_{n_2,p_2}$ with $p_2>\frac{n_2}{2},$ we have $v_{p_2}(xy)=1.$ This means that $p_2$ divides exactly one of $x$ and $y.$ Then $v_{p_2}(x+y)=0.$ Therefore $p_2>p_1$ as $x+y=X_{n_1,p_1}.$
 
 Since $p_1<p_2 ,$  we have $v_{q}(x+y)\geq 1$ and $v_q(xy)\geq 1$ for each prime 
 $q\leq p_1.$ This gives $v_q(x)\geq 1$ and $v_q(y)\geq 1$ for each prime $q\leq p_1.$ 
Then there exists $T(x,y)\in \mathbb{N}$ such that
  \begin{equation} 
  x+y=T(x,y)\displaystyle \prod_{q\leq p_1 } q, \label{T_(1)}
  \end{equation}
\begin{align}
\text{ where }T(x,y)=\left(\prod_{ q\leq p_1 } q^{v_q(x)-1}\right) \left(\prod_{p_1<q} q^{v_q(x)}\right)
+ \left(\displaystyle \prod_{q\leq p_1 } q^{v_q(y)-1}\right)\left(\prod_{p_1<q}q^{v_q(y)}\right).\nonumber\label{T_(2)}
\end{align}
Since $xy=X_{n_2,p_2},$ we have $p_2\mid x$ or $p_2\mid y$ and hence   $T(x,y)\geq  p_2$, as $p_2>p_1.$

Using the fact that $x+y=X_{n_1, p_1}\in \mathcal{X}$ along with the definition of $\mathcal{X}$ as stated in Remark \ref{remark1}, eq.\eqref{T_(1)}  gives us \begin{equation}
T(x, y)\prod_{q \in W(T(x, y))}q < 2p_1. \label{T'eq} \end{equation}
Since $T(x, y) >1,$ there exists a prime $q_0$ such that $q_0 \mid T(x, y).$ So, the inequality in  \eqref{T'eq}  gives
 $2T(x, y) \leq q_0T(x, y) < 2p_1$ i.e. $T(x, y) < p_1<p_2,$ a contradiction to  the fact that $T(x,y)\geq p_2.$

Hence, there exist no $x, y \in \mathbb{N}$ such that $x+y, xy \in \mathcal{X}.$
 \end{proof}
Now, we look for arithmetic and geometric progressions in $N_1.$
By Propopsition \ref{piecewise_syndetic_N_1} and a result in  \cite{hindman} stated below, we conclude that $N_1$ contains 
arbitrarily long arithmetic and geometric progressions.
\begin{proposition}[Hindman \cite{hindman}, p. 361]
 Let $A$ be a multiplicatively piecewise syndetic subset of $\mathbb{N}$ and $k\in\mathbb{N}.$ Then there exists 
 $a,d\in \mathbb{N}$ and $r\in \mathbb{N}\setminus \{1\}$ such that \{$r^j(a+id)\colon i,j\in\{0,1,\dots,k\}\}\cup \{dr^j\colon j\in\{0,1,\dots,k\}\}\subset A.$
\end{proposition}

\begin{proposition}\label{geo}
$N_1$ contains arbitrarily long arithmetic, geometric and geo-arithmetic progressions.
\end{proposition}
Using the ideas of Proposition \ref{XIp0}, we will now give explicit examples of infinite families of arithmetic and 
geometric progressions in $N_1$ via $\mathcal{X}.$ Suppose $\{m_1, m_2, \dots, m_n\}$ is an arithmetic (or geometric) progression of length $n$ in
$\mathbb{N}.$ Let $b_i=m_i \prod_{q\in W(m_i)}q$ for each $1\leq i\leq n$ and let $p>\max\{\frac{b_i}{2}\colon 1\leq i\leq n\}.$
Then $\{X_{b_i, p}\}_{i=1}^{n}$ is an arithmetic (or geometric) progression in $\mathcal{X}$ of length $n.$

Combining Propositions \ref{piecewise_syndetic_N_1}, \ref{XIp0}, \ref{additive_IP},  \ref{multiplicative_IP}, \ref{x+y_xy} and \ref{geo}, we get Theorem \ref{t5}.

\section{\textbf{Questions}}\label{section_question}
 In this section, we state some problems on sparsely totient numbers  arising from the present work. 
 
For a prime $p,$ define $TN_1(p) :=\inf\{m\in V\colon k> m\Rightarrow p\mid N_1(k)\}.$  For example, $TN_1(2)=1$ and
Corollary \ref{t1corr1} gives $TN_1(3)=2.$ Theorem \ref{t1} ensures the existence of $TN_1(p)$ for each $p$. One can observe that
$TN_1(p)=\max\{m\in V\colon p\nmid N_1(m)\}.$
 Proposition \ref{p1} and Proposition \ref{p2} give us that if $p\nmid N_1(k),$ then $v_q(N_1(k))=0$ for $q>p(p-1)$, $v_q(N_1(k))\leq 2$ for $p<q\leq p(p-1),$
and $v_q(N_1(k))\leq D_p(q)$ for $q<p$, where $D_p(q)$ is independent of $k.$ Here we do not know quantitative information
about $D_p(q)$. A knowledge of this will give an explicit upper bound of $TN_1(p)$. These observations 
give rise to the following problems:

\begin{question}\label{sectionq1}
 For a prime $p,$ what is the value of $TN_1(p)$?
\end{question}
\begin{question}
 Let $p$ and $q$ be primes such that $q<p$. If $N$ is a sparsely totient number such that each prime factor of $N$ is
 less than $p,$ then find an explicit upper bound for $v_q(N)$.
\end{question}

Let  $m_1, m_2,\dots, m_k, \dots $ be the  
enumeration of  the elements of the set $BN_1,$ where $BN_1=\{m \in V : N_1(m) = \max(\phi^{-1}(m))\}$.
From the table of values of $N_1(m)$ below, it seems that $N_1(m_k)> k^2$ for each $k.$
Can we prove  $N_1(m_k)> k^{1+\epsilon},$ where $\epsilon$ is a small positive real number. This would also mean that $\sum_{m\in BN_1}(N_1(m))^{-1}$ is convergent. 
\begin{center}\begin{tabular}{|c|c|c|c|c|c|c|c|c|c|c|c|c|c|c|c|c|c|}
 \hline
 $k$& 1&2 &3 &4 &5 &6 &7 & 8&9 &10 &11 &12 &13    \\
 \hline
 $N_1(m_k)$&2& 6&12 &18&30 &42&60 &66&90 & 120&126 &150 &210    \\
 \hline
 $k^2 $&1 &4 & 9& 16& 25& 36& 49& 64& 81& 100& 121& 144& 169  \\
 \hline
\end{tabular}\end{center}
Hence we ask the following question.
\begin{question}\label{sectionq2}
Does there exist $\epsilon>0$ such that $N_1(m_k)>k^{1+\epsilon}$ for each $k\in\mathbb{N}$?
Does $\sum_{m\in BN_1}(N_1(m))^{-1}$ converge?\end{question}
As discussed in Section \ref{introduction} via observations from  Figure \ref{plot}, we raise the following two questions:  
\begin{question} \label{sectionq3}
 Let $p,p_1$ and $p_2$ be consecutive primes for each $p\in\mathbb{P}.$ Does the fractional part of $\frac{p_1p_2}{p}$ tend to $0$ as $p\rightarrow \infty$?  
\end{question}

 \begin{question}\label{sectionq4}
 Is $\mathcal{Y}_2=\{Y_{p,2}\colon p\in \mathbb{P}, p\geq 11\}$? Is the set $\mathcal{Y}_3$ infinite? 
 \end{question}
 
As we have seen in Proposition \ref{x+y_xy}, $\mathcal{X}$ doesn't contain a configuration of the type $\{x+y, xy\}.$ So, the question
arises whether such a configuration is possible in other families of subsets of $N_1.$ More generally,
\begin{question}\label{sectionq5}
 Does $N_1$ contain the set $\{x+y, xy\}$ for some $x, y \in \mathbb{N}$?
\end{question}
  
   To the best of the authors' knowledge, all the above problems are open. 

\section*{Acknowledgement}
We would like to thank the Department of Atomic Energy, Government of India for the financial support. 
In addition, the research of authors Bhuwanesh Rao Patil and Pramod Eyyunni was also supported by the `INFOSYS scholarship for senior students'.
We would also like to thank Harish-Chandra Research Institute for the excellent facilities.
We sincerely thank  Samrat Kadge for help with the 
computation and programming. We thank the referee for a  careful reading of this paper and suggestions.

\end{document}